\documentclass[11pt, reqno]{amsart}    

\usepackage{amsmath,amssymb,latexsym}
\usepackage{color}
\usepackage{xcolor}
\usepackage{graphicx}
\usepackage{cite}

\newcommand{\R}{\mathbb{R}}

\newcommand{\Z}{\mathbb{Z}}
\newcommand{\N}{\mathbb{N}}

\newcommand{\bigO}{\mathcal{O}}
\newcommand{\diff}{\mathrm{d}}

\DeclareMathOperator{\supp}{supp}
\newcommand{\abs}[1]{\lvert#1\rvert}
\newcommand\norm[1]{\left\lVert#1\right\rVert}
\newtheorem{theorem}{Theorem}[section]
\newtheorem{lemma}[theorem]{Lemma}
\newtheorem{proposition}[theorem]{Proposition}
\newtheorem{remark}[theorem]{Remark}

\newtheorem*{main-theorem}{Main Theorem}
\newtheorem*{remark*}{Remark}
\newtheorem*{lemma*}{Lemma A.1}

\numberwithin{equation}{section}

\begin{document}

\title[Solitary wave solutions of Whitham--Boussinesq systems]{Solitary wave solutions to a class of Whitham--Boussinesq systems}

\author{Dag Nilsson}
\address{Department of Mathematical Sciences, Norwegian University of Science and Technology, 7491 Trondheim, Norway.}
\email{dag.nilsson@ntnu.no}

\author{Yuexun Wang}
\address{Department of Mathematical Sciences, Norwegian University of Science and Technology, 7491 Trondheim, Norway.}
\email{yuexun.wang@ntnu.no}

\thanks{D.N.was supported by an ERCIM `Alain Bensoussan' Fellowship. Y.W. acknowledges the support by grants nos. 231668 and 250070 from the Research Council of Norway.}
 

\begin{abstract}
In this note we study solitary wave solutions of a class of Whitham--Boussinesq systems which includes the bi-directional Whitham system as a special example. The travelling wave version of the evolution system can be reduced to a single evolution equation, similar to a class of equations studied by Ehrnstr{\"o}m, Groves and Wahl{\'e}n \cite{EGW}. In that paper the authors prove the existence of solitary wave solutions using a constrained minimization argument adapted to noncoercive functionals, developed by Buffoni \cite{Buffoni2004}, Groves and Wahl{\'e}n \cite{MR2847283}, together with the concentration-compactness principle.
%concentration-compactness principle adapted to noncoercive functionals developed in \cite{buffoni2004,MR2847283}.    
\end{abstract}

\maketitle
\section{Introduction}\label{Introduction}
This work is devoted to the study of solitary wave solutions  of the Whitham--Boussinesq system
\begin{equation}\label{eq:bdw}
\begin{aligned}
\partial_t\eta &=-K\partial_xu-\partial_x(\eta u)\\
\partial_tu &=-\partial_x\eta-u\partial_xu.
\end{aligned}
\end{equation}
A solitary wave is a solution of the form  
\begin{align}\label{eq:travelling wave ansatz}
\eta(x,t)=\eta(x-ct),\  u(x,t)=u(x-ct),
\end{align}
such that \(\eta(x-ct),u(x-ct)\longrightarrow 0\) as \(|x-ct|\longrightarrow \infty\).
Here, $\eta$ denotes the surface elevation, \(u\) is the rightward velocity at the surface, and $K$ is a Fourier multiplier operator  defined by
\begin{equation*}\label{eq:K}
\mathcal{F}(Kf)(k)=m(k)\hat{f}(k),
\end{equation*}  
for all \(f\) in the Schwartz space \(\mathcal{S}(\mathbb{R})\). More specifically, we require that
\begin{itemize}
\item[(A1)] The symbol $m\in S_\infty^{m_0}(\R)$ for some $m_0<0$, that is
\begin{align*}
|m^{(\alpha)}(k)|\leq C_\alpha(1+|k|)^{m_0-\alpha},\  \alpha\in \mathbb{N}_0. 
\end{align*}
\item[(A2)] The symbol \(m:\R\rightarrow\R\) is even and satisfies \(m(0)>0\), $m(k)<m(0),\  \text{for}\ k\neq 0$ and 
\begin{align*}
m(k)=m(0)+\frac{m^{(2j_*)}(0)}{(2j_*)!}k^{2j_*}+r(k),
\end{align*}                                                                                                                                                                                                         
for some \(j_*\in \mathbb{N}_+\), where \(m^{(2j_*)}(0)<0\) and \(r(k)=\mathcal{O}(k^{2j_*+2})\) as \(k\rightarrow 0\).\\
\end{itemize}
As an example we have $m(k)=\tanh(k)k^{-1}$, which yields the bi-directional Whitham (BDW) system, and this choice of symbol is the main motivation for studying \eqref{eq:bdw}. The BDW system was formally derived in \cite{MR2991247,MR3390078} from the incompressible Euler equations to model fully dispersive shallow water waves whose propagation is allowed to be both left- and rightward,  and appeared in \cite{MR3060183,MR3668593} as a full dispersion system in the Boussinesq regime with the dispersion of the water
waves system. There have been several investigations on the BDW system: local well-posedness \cite{2017arXiv170804551E,MR3763731} (in homogeneous Sobolev spaces at a positive background), a logarithmically cusped wave of greatest height \cite{Ehrnström2018}. There are also numerical results, investigating the validity of the BDW system as a model of waves on shallow water \cite{MR3844340}, numerical bifurcation and spectral stability \cite{doi:10.1111/sapm.12221} and the observation of dispersive shock waves 
%developing development of shock waves 
\cite{MR3523508}. However there are no results on the existence of solitary wave solutions. 

We also mention that one can include the effects of surface tension in the BDW system by choosing $m(k)=\tanh(k)k^{-1}(1+\beta k^2), \beta>0$. It was recently shown in \cite{2018arXiv180504372K} that \eqref{eq:bdw} is locally well-posed for this choice of symbol. However, the above symbol with $\beta>0$ is not included in the class of symbols considered in the present work.          
%An example is $m(k)=\frac{\tanh(k)}{k}$, which yields the bi-directional Whitham equations. 
Moreover, in \cite{MR3749383, DDK, ED}, other types of fully dispersive Whitham-Boussinesq systems are considered. We also mention the generalized class of Green--Nagdhi equations introduced in \cite{MR3564304}, which was shown to posses solitary wave solutions in \cite{MR3841973}.

\section{Solitary wave solutions to the Whitham equation}
In order to prove existence of solitary wave solutions of \eqref{eq:bdw} our strategy will be to reduce this to a problem that is similar to one studied in \cite{EGW}. For this reason we first discuss the results and methods of that paper.
In \cite{EGW} the authors prove the existence of solitary wave solutions of the pseudodifferential equation
\begin{equation}\label{EGW-eq}
u_t+\big(Ku+\tilde{n}(u)\big)_x=0,
\end{equation}
where $K$ have properties $(A1)$, $(A2)$ and the nonlinearity $\tilde{n}$ satisfies
\begin{itemize}
\item[(A3)] The nonlinearity $\tilde{n}$ is a twice continuously differentiable function $\mathbb{R}\rightarrow \mathbb{R}$ with
\begin{equation*}\label{egw}
\tilde{n}(x)=\tilde{n}_p(x)+\tilde{n}_r(x),
\end{equation*}
in which the leading order part of the nonlinearity takes the form $\tilde{n}_p(x)=c_p\abs{x}^p$ for some $c_p\neq 0$ and $p\in [2,4j_*+1)$ or $\tilde{n}_p(x)=c_px^p$ for some $c_p>0$ and odd integer $p$ in the range $p\in [2,4j_*+1)$, while
\begin{equation*}
\tilde{n}_r(x)=\mathcal{O}(\abs{x}^{p+\delta}),\quad \tilde{n}_r'(x)=\mathcal{O}(\abs{x}^{p+\delta-1})
\end{equation*}
for some $\delta>0$ as $x\rightarrow 0$.
\end{itemize}
In particular, the uni-directional Whitham equation, introduced in \cite{MR0671107}, belongs to this class of equations \eqref{EGW-eq},  with $m(k)=\sqrt{\tanh(k)k^{-1}}$. The Whitham equation possesses periodic travelling waves \cite{MR2555644} and solitary waves \cite{EGW}, moreover the solitary waves decay exponentially \cite{MR3603270}. 
%We mention here that the solitary wave solutions of the Whitham equation were shown to decay exponentially \cite{MR3603270}. 
%Apart from solitary wave solutions, the Whitham equation has also been shown to posses periodic travelling waves \cite{MR2555644}. 
It was recently confirmed that the Whitham equation possesses a highest cusped wave \cite{2016arXiv160205384E}, as conjectured by Whitham.

Under the travelling wave ansatz: $u(t,x)=u(x-c t)$, the equation \eqref{EGW-eq} becomes
\begin{equation}\label{travellingwave-EGW}
Ku-c u+\tilde{n}(u)=0.
\end{equation}
The existence of solutions of \eqref{travellingwave-EGW} is established via a related minimization problem. Let
\begin{equation*}
\tilde{\mathcal{E}}(u)=-\frac{1}{2}\int_\mathbb{R}uKu\ \mathrm{d}x-\int_\mathbb{R}\tilde{N}(u)\ \mathrm{d}x,\quad \mathcal{I}(u)=\frac{1}{2}\int_\mathbb{R}u^2\ \mathrm{d}x
\end{equation*}
with 
\begin{align*}
\tilde{N}(x)&=\tilde{N}_{p+1}(x)+\tilde{N}_r(x),\\
\tilde{N}_{p+1}(x)&=\int_0^x\tilde{n}_p(s)\ \mathrm{d}s=\frac{ c_px^{p+1}}{p+1}, \text{ or }\frac{ c_px\abs{x}^{p}}{p+1}, \\
\tilde{N}_r(x)=&\int_0^x\tilde{n}_r(s)\ \mathrm{d}s=\mathcal{O}(\abs{x}^{p+1+\delta}).
\end{align*}
Let $q,R>0$ and
\begin{equation*}
V_{q,R}:=\{u\in H^1(\mathbb{R}):\ \mathcal{I}(u)=q,\ \norm{u}_{H^1}<R\}.
\end{equation*}
Minimizers of $\tilde{\mathcal{E}}$ over $V_{q,R}$ (that are not on the boundary) satisfy the Euler-Lagrange equation
\begin{equation}\label{EL-egw}
\mathrm{d}\tilde{\mathcal{E}}(u)+\nu\mathrm{d}\mathcal{I}(u)=0,
\end{equation}
for a Lagrange multiplier $\nu$, and \eqref{EL-egw} is precisely \eqref{travellingwave-EGW}, with $c=\nu$. In \cite{EGW} the authors show that there exist solutions of the minimization problem
\[
\arg \underset{V_{q,R}} \inf  \tilde{\mathcal{E}}(u),
\]
which by the above argument yields travelling wave solutions of \eqref{EGW-eq}.
%%%%%%%%%% Fortsätt här. Ange referenser på ett snyggare sätt. Lägg eventuellt till Mathias artikel.
The existence of minimizers is established using methods developed in \cite{Buffoni2004,MR2847283} and we give here a brief outline of the proof. The functional $\tilde{\mathcal{E}}$ is not coercive and since the domain is unbounded one cannot use the Rellich--Kondrachov theorem. In particular, direct methods cannot be used to obtain a minimizer. Because of this one needs to study a related penalized functional acting on periodic functions. 
Let \(P>0\) and \(L_P^2\) be the space of \(P\)- periodic, locally square-integrable functions with Fourier-series representation 
\[
w(x)=\frac{1}{\sqrt{P}}\sum_{k\in\Z}\widehat{w}(k)\exp(2\pi ikx/P),
\]
with 
\[
\widehat{w}(k):=\frac{1}{\sqrt{P}}\int_{-\frac{P}{2}}^{\frac{P}{2}}w(x)\exp(-2\pi ikx/P)\,\diff x.
\]
For \(s\geq0\),  we define  
\[
H^s_P:=\{w\in L^2_P:\ \|w\|_{H^s_P}<\infty\},
\]
where the norm is given by 
\[
\|w\|_{H^s_P}:=\left(\sum_{k\in \Z}\bigg(1+\frac{4\pi^2k^2}{P^2}\bigg)^s|\widehat{w}(k)|^2 \right)^{\frac{1}{2}}.
\]
The authors \cite{EGW} studied the following penalized functional 
\begin{equation*}
\tilde{\mathcal{E}}_{P,\varrho}(u):=\varrho(\norm{u}_{H_P^1}^2)+\tilde{\mathcal{E}}_p(u),
\end{equation*}
over the set
\begin{equation*}
V_{P,q,R}:=\{u\in H_P^1:\ \mathcal{I}_P(u)=q, \ \norm{u}_{H_P^1}<2R\},
\end{equation*}
where $\tilde{\mathcal{E}}_P, \ \tilde{\mathcal{I}}_P$ are the same functionals as $\tilde{\mathcal{E}},\ \tilde{\mathcal{I}}$ but where the integration is over $[-P/2,P/2]$, and $\varrho:[0,(2R)^2]\mapsto [0,\infty)$ is a penalization function such that $\varrho(t)=0$ whenever $t\in[0,R^2]$ and $\varrho(t)\rightarrow \infty$ as $t\rightarrow (2R)^2$. The penalization function makes $\tilde{\mathcal{E}}_{P,\varrho}$ coercive, and the fact that we are now working in $H_P^1$ allows the use of the Rellich-Kondrachov theorem. It is then an easy task to show that there exists a minimizer $u_P\in V_{P,q,2R}$, of $\tilde{\mathcal{E}}_{P,\varrho}$. The next step is to show that $u_P$ in fact minimizes $\tilde{\mathcal{E}}_P$ over $V_{q,R}$. This is immediate after showing that
\begin{equation*}\label{main-ineq}
\norm{u_P}_{H_P^1}^2\leq Cq,
\end{equation*}
and choosing $q$ sufficiently small. The other key ingredient of the proof is the concentration compactness theorem \cite{MR778974}. In the application of this theorem, the main task is to show that `dichotomy' does not occur. This is done using proof by contradiction, where the contradiction is arrived at using the strict subadditivity of
\begin{equation*}
I_q:=\arg\underset{V_{q,R}}\inf\tilde{\mathcal{E}}(u),
\end{equation*}
as a function of $q$. The strict subadditivity of $I_q$ is established by using a special minimizing sequence for $\tilde{\mathcal{E}}$, constructed from the minimizers $u_P$. In addition it is necessary to decompose $u$ into high and low frequencies in order to get satisfactory estimates on $\norm{u}_{L^\infty}$, see \cite[Corollary 4.5]{EGW}.
It is an easy task to show that `vanishing' cannot occur either. Therefore, from the concentration compactness theorem, `concentration' is the only possibility and the existence of minimizers then follows from a standard argument.

Under the additional assumption that 
\begin{itemize}
	\item[(A4)] $\tilde{n}\in C^{2j_*}(\mathbb{R})$ with 
	\[\tilde{n}_r^{(j)}(x)=\mathcal{O}(\abs{x}^{p+\delta-j}),\ j=0,\ldots, ,2j_*,\] 
\end{itemize} 
it is possible to relate the minimizers of $\tilde{\mathcal{E}}$ to those of $\tilde{\mathcal{E}}_{lw}$, where
\begin{equation*}
\tilde{\mathcal{E}}_{lw}(u)=-\int_\mathbb{R}\left(\frac{m^{(2j_*)}(0)}{2(2j_*)!}(u^{(j_*)})^2+ \tilde{N}_{p+1}(u)\right)\, \mathrm{d}x. 
\end{equation*} 
More specifically,
\begin{equation*}
\sup_{u\in \tilde{D}_q}\text{dist}_{H^{j_*}(\mathbb{R})}(S_{lw}^{-1}u,\tilde{D}_{lw})\rightarrow 0,\quad \text{ as }q\rightarrow 0,
\end{equation*}
where $\tilde{D}_{lw}$ is the set of minimizers of $\tilde{\mathcal{E}}_{lw}$ over the set 
\[\{u\in H^{j_*}(\mathbb{R}):\ \mathcal{I}(u)=1\},\]
and \(\tilde{D}_q\) is the set of minimizers of $\tilde{\mathcal{E}}$ over $V_{q,R}$ and 
\[(S_{lw}u)(x):=q^\alpha u(q^\beta x)\] 
is the 'long-wave test function' with 
\begin{equation}\label{alpha-beta}
\alpha=\frac{2 j_*}{4j_*+1-p},\quad \beta=\frac{p-1}{4j_*+1-p}.
\end{equation}
The numbers $\alpha$ and $\beta$ are chosen in such a way that 
\[2\alpha-\beta=1,\quad (p-1)\alpha=2j_*\beta.\] 
This choice of $\alpha$, $\beta$ appear naturally when deriving the long-wave approximation of \eqref{travellingwave-EGW}. 
The functional $\tilde{\mathcal{E}}_{lw}$ is related to $\tilde{\mathcal{E}}$ via (see \cite[Lemma 3.2]{EGW})
\begin{equation*}
\tilde{\mathcal{E}}(S_{lw}u)=-qm(0)+q^{1+(p-1)\alpha}\tilde{\mathcal{E}}_{lw}(u)+o(q^{1+(p-1)\alpha}), 
\end{equation*}
for any \(u\in W:=\{u\in H^{2j_*}(\R):\ \norm{u}_{H^{2j_*}}<S\}\) with \(S\) being a positive constant. 

We mention here a recent work \cite{arXiv:1802.10040} where they use an entirely different approach to prove the existence of small amplitude solitary wave solutions of the Whitham equation.
\section{Solitary wave solutions to the Whitham--Boussinesq system}\label{sec-bi-directional-whitham}
\subsection{Formulation as a constrained minimization problem}
In the present work we seek solitary wave solutions of \eqref{eq:bdw}, and the idea is to reformulate \eqref{eq:bdw} in such a way that the method of \cite{EGW} can be applied. Under the travelling wave ansatz \eqref{eq:travelling wave ansatz}, the system \eqref{eq:bdw} then becomes 
\begin{align}
c\eta &=Ku+\eta u,  \label{2}\\
cu &=\eta+\frac{u^2}{2}   \label{3}.
\end{align}
It follows from \eqref{3} that \(\eta=u(c-\frac{u}{2})\), and if we insert this into \eqref{2} then we find that 
\begin{equation}\label{4}
Ku-u(u-c)(\frac{u}{2}-c)=0.
\end{equation}

We first formally assume that \(\|u\|_{L^\infty}\ll c \) to formulate \eqref{4} into a variational problem. This is no restriction since the constructed solutions will automatically satisfy this smallness condition (see Theorem \eqref{th:main}).
Let \(w=\frac{u}{c}(\frac{u}{c}-2)\), so that \(u=c-c\sqrt{1+w}\). The map \(w \mapsto u\) is well-defined, since
\begin{align*}
\norm{w}_{L^\infty}\leq \norm{\frac{u}{c}}_{L^\infty}\norm{\frac{u}{c}-2}_{L^\infty}\lesssim \norm{\frac{u}{c}}_{L^\infty}\ll1,\\
\end{align*}
%and
%\begin{align*}
%\norm{\partial_xw}_{L^2(\R)}&\leq \norm{\frac{\partial_xu}{c}}_{L^2(\R)}\norm{\frac{u}{c}-2}_{L^\infty(\R)}+\norm{\frac{u}{c}}_{L^\infty(\R)}\norm{\frac{\partial_xu}{c}}_{L^2(\R)}\\
%&\lesssim \norm{\frac{\partial_xu}{c}}_{L^2(\R)}\ll 1. 
%\end{align*} 
We then may rewrite the equation \eqref{4} using the new unknown \(w\) as 
\begin{align}\label{vareq}
\frac{2}{\sqrt{1+w}}K(\sqrt{1+w}-1)-\lambda w=0,
\end{align}
with $\lambda=c^2$.
%Introduce the function \(v\) by writing \(u=cv\). Note that \(\|v\|_{H^1(\R)}\ll 1\) since \(\|u\|_{H^1(\R)}\ll c\). We then may rewrite the equation \eqref{4} by the new variable \(v\) as 
%\begin{equation}\label{5}
%Kv-c^2v(v-1)(\frac{v}{2}-1)=0.
%\end{equation}
%Let \(w=v(v-2)\), so that \(v=1-\sqrt{1+w}\). In view of \(\|v\|_{H^1(\R)}\ll 1\), one has \(\|w\|_{H^1(\R)}\ll 1\). In fact, 
%\begin{align*}
%\|w\|_{L^2(\R)}\leq \|v-2\|_{L^\infty(\R)}\|v\|_{L^2(\R)}\lesssim \|v\|_{L^2(\R)}\ll1,\\
%\end{align*}
%and
%\begin{align*}
%\|\partial_xw\|_{L^2(\R)}&\leq \|v-2\|_{L^\infty(\R)}\|\partial_xv\|_{L^2(\R)}+\|v\|_{L^\infty(\R)}\|\partial_xv\|_{L^2(\R)}\\
%&\lesssim \|\partial_xv\|_{L^2(\R)}\ll1. 
%\end{align*}
%Then by Sobolev embedding \(H^1(\R)\hookrightarrow L^\infty(\R)\), we see that the map \(w \mapsto v\) is well-defined. Therefore, \(w\) satisfies 
%\begin{align}\label{vareq}
%\boxed{\frac{2}{\sqrt{1+w}}K(\sqrt{1+w}-1)-\lambda w=0},
%\end{align}
%with $\lambda=c^2$.
We now define 
\begin{equation*}
\mathcal{E}(w)=\underbrace{-\frac{1}{2}\int_\mathbb{R}wKw\ \mathrm{d}x}_{:= \mathcal{K}(w)}\underbrace{-\int_\mathbb{R}N(w)\ \mathrm{d}x}_{:=\mathcal{N}(w)},
\end{equation*}
where
\begin{align*}
N(w)=2\Psi(w)Kw+2\Psi(w)K(\Psi(w)), 
\end{align*}
\begin{align*}
 \Psi(w)=\sqrt{1+w}-1-\frac{w}{2}=-\frac{w^2}{8}+\Psi_r(w),
\end{align*}
\begin{align*}
\ \Psi_r(x)=\mathcal{O}(x^3).
\end{align*}
To extract the lower-order parts we also write 
\begin{equation*}
N(w)=N_h(w)+N_l(w),
\end{equation*}
with 
\begin{equation*}
N_h(w)=-\frac{w^2}{4}Kw,\quad N_l(w)=2\Psi(w)Kw+2\Psi(w)K(\Psi(w)).
\end{equation*}
%It is also convenient to let 
%\begin{equation*}
%\mathcal{K}(w)=-\frac{1}{2}\int_{\mathbb{R}}wKw\ \mathbb{d}x, \ \mathcal{N}(w)=-\int_{\mathbb{R}}N(w)\ \mathrm{d} x.
%\end{equation*}
We then note that
\[
\diff \mathcal{E}(w)+\lambda \diff \mathcal{I}(w)=0
\]
is precisely \eqref{vareq}. Hence, \(w\) is a critical point of \(\mathcal{E}\) under the constraint \(\mathcal{I}(w)=q\), if and only if \(u=c-c\sqrt{1+w}\) is a solution of \eqref{4}, with \(\lambda=c^2\). 
We will find critical points of 
\(\mathcal{E}(w)+\lambda  \mathcal{I}(w)\) by considering the minimization problem 
\[
\arg \underset{V_{q,R}} \inf  \mathcal{E}(w).
\]
Here we are minimizing a functional $\mathcal{E}$ of almost the same type as in \cite{EGW}, with $p=2$, but with a slightly different nonlinearity. In our case, the nonlocal operator $K$ appears in the nonlinear term $N$. However, since $K$ is a bounded smoothing operator, it is not hard to show that the methods used in \cite{EGW} can be applied to the functional $\mathcal{E}$. However, the results \cite[Lemma 2.3, Lemma 3.2, Lemma 3.3]{EGW} require a bit more care, in particular it is important to know how $\mathcal{N}$ acts under the long-wave scaling, and we therefore include the proofs of these results in the next subsection. We finally have the following existence result:
\begin{theorem}\label{th:main}  
There exists \(q_*>0\) such that the following statements hold for each \(q\in(0,q_*)\).

(\(\romannumeral1\)) The set \(D_q\) of minimizers of \(\mathcal{E}\) over the set \(V_{q,R}\) is nonempty and the estimate \(\|w\|_{H^1(\R)}^2=\bigO(q)\) holds uniformly over \(w\in D_q\). Each element of \(D_q\) is a solution of the
travelling wave equation \eqref{vareq}; the squared wave speed \(c^2\) is the Lagrange multiplier in this constrained
variational principle. 

(\(\romannumeral2\)) Let \(s<1\) and suppose that \(\{w_n\}_{n\in\mathbb{N}_0}\) is a minimizing sequence for \(\mathcal{E}\) over \(V_{q,R}\). 
There exists a sequence \(\{x_n\}_{n\in\mathbb{N}_0}\) of real numbers such that a subsequence of \(\{w_n(\cdot+x_n)\}_{n\in\mathbb{N}_0}\) converges in \(H^s(\R)\) to a function in \(D_q\). 
\end{theorem}

\subsection{Technical results}
In our case the long-wave functional $\mathcal{E}_{lw}$ is given by
\begin{align*}
\mathcal{E}_{lw}(w):=-\int_\R\left(\frac{m^{2j_*}(0)}{2(2j_*)!}(w^{(j_*)})^2-\frac{m(0)}{4}w^3\right)\diff x,
\end{align*}
and we also recall the long-wave scaling: 
\[S_{lw}w(x)=\mu^\alpha w(\mu^\beta x),\]
with
\begin{equation}\label{alphabeta_specialcase}
\alpha=\frac{2j_*}{4j_*-1} \quad \text{and} \quad \beta=\frac{1}{4j_*-1}.
\end{equation}
Note that \eqref{alphabeta_specialcase} is a special case of \eqref{alpha-beta}, with $p=2$.

We first present a result corresponding to \cite[Lemma 3.2]{EGW}, which relates $\mathcal{E}$ with $\mathcal{E}_{lw}$.
\begin{lemma}\label{le:long wave}  Let  $w\in W$ with $\norm{w}_{L^\infty}\ll 1$ and $\mathcal{I}(w)=1$. Then
	\begin{align}\label{eq:long wave test}
	\mathcal{E}(S_{lw}w)=-q m(0)+q^{1+\alpha}\mathcal{E}_{lw}(w)+o(q^{1+\alpha}). 
	\end{align}
\end{lemma}

\begin{proof} Recall the definition 
\begin{equation*}
\mathcal{E}(S_{lw}w)=\mathcal{K}(S_{lw}w)+\mathcal{N}(S_{lw}w).
\end{equation*}
We first calculate that
\begin{align*}
&\mathcal{K}(S_{lw}w)\\
&=-\frac{1}{2}\int_{\mathbb{R}}q^{2\alpha}w(q^\beta x)Kw(q^\beta \cdot)(x)\, \mathrm{d}x\\
&=-\frac{1}{2}\int_\mathbb{R}q^{2\alpha}m(k)\abs{\mathcal{F}(w(q^\beta\cdot))(k)}^2\, \mathrm{d}k\\
&=-\frac{1}{2}\int_\mathbb{R}q^{2\alpha-\beta}\left(m(0)+q^{2j_*\beta}\frac{m^{(2j_*)}(0)}{(2j_*)!}k^{2j_*}+r(q^\beta k)\right)\abs{\hat{w}(k)}^2\ \mathrm{d}k\\
&=-q m(0)-q^{2\alpha+(2j_*-1)\beta}\int_\mathbb{R}\frac{m^{(2j_*)}(0)}{2(2j_*)!}(w^{j_*})^2\ \mathrm{d}x-\frac{q^{2\alpha-\beta}}{2}\int_\mathbb{R}r(q^\beta k)\abs{\hat{w}(k)}^2\, \mathrm{d}k,
\end{align*}
and one may continuously estimate the last term as 
\begin{equation*}
\abs{\frac{q^{2\alpha-\beta}}{2}\int_\mathbb{R}r(q^\beta k)\abs{\hat{w}(k)}^2\, \mathrm{d}k}\lesssim q^{2\alpha+(2j_*+1)\beta}\int_\mathbb{R}k^{2j_*+2}\abs{\hat{w}(k)}^2\, \mathrm{d}k,
\end{equation*}
and $\int_\mathbb{R}k^{2j_*+2}\abs{\hat{w}(k)}^2\ \mathrm{d}k$ is uniformly bounded, since $w\in W$. We next consider
\begin{equation*}
\mathcal{N}(S_{lw}w)=-\int_\mathbb{R}N_h(S_{lw}w)+N_l(S_{lw}w)\, \mathrm{d}x.
\end{equation*}
A direct calculation shows that
\begin{align*}
&-\int_\mathbb{R}N_h(S_{lw}w)\ \mathrm{d}x=\int_{\mathbb{R}}\frac{q^{3\alpha}}{4}w^2(q^\beta x)Kw(q^\beta\cdot)(x)\ \mathrm{d}x\\
&=\int_\mathbb{R}\frac{q^{3\alpha-\beta}}{4}\overline{\mathcal{F}(w^2)(k)}\hat{w}(k)\left(m(0)+q^{2j_*\beta}\frac{m^{(2j_*)}(0)}{(2j_*)!}k^{2j_*}+r(q^\beta  k)\right)\, \mathrm{d}k\\
&=q^{3\alpha-\beta}\int_{\mathbb{R}}\frac{m(0)}{4}w^3\, \mathrm{d}x+o(q^{3\alpha-\beta}),
\end{align*}
where we again used that $w\in W$ in order to estimate the remaining terms.
The term $\int_\mathbb{R}N_l(S_{lw}w)\ \mathrm{d}x$ is of lower order and can be estimated in the same way.

Combining all the above estimates yields the identity \eqref{eq:long wave test}.
	%\begin{equation}
	%\begin{aligned}\label{eq:*}
	%\mathcal{E}(w)&%=-2\int_\R (\frac{1}{2}w-\frac{1}{4}w^2+\cdots)\mathcal{K}(\frac{1}{2}w-\frac{1}{4}w^2+\cdots)\diff x\\
	%&=-2\int_\R \mathcal{F}(\frac{1}{2}w-\frac{1}{4}w^2+\cdots)(k)\bigg(m(0)+\frac{m^{(2j_*)}(0)}{(2j_*)!}k^{2j_*}+r(k)\bigg)\\
	%&\quad\quad\quad\quad\times\overline{\mathcal{F}(\frac{1}{2}w-\frac{1}{4}w^2+\cdots)}(k)\diff k\\
	
	%=-\int_\R \frac{1}{2}m(0)w^2dx-\int_\R\left(\frac{m^{2j_*}(0)}{2(2j_*)!}(w^{(j_*)})^2-\frac{1}{4}m(0)w^3\right)\diff x\\
	%&\quad+\int_\R \bigg[-\frac{m^{(2j_*)}(0)}{4(2j_*)!}k^{2j}\big(\widehat{w}(k)\overline{\widehat{w^2}}(k)+\widehat{w^2}(k)\bar{\hat{w}}(k)\big)+\cdots\bigg]\diff k. 
	%\end{aligned}
	%\end{equation}
	%Then \eqref{eq:long wave test} is an easy scaling argument together with Sobolev embedding by using \(S_{lw}w\) instead of \(w\) in \eqref{eq:*}. 	
\end{proof}

We next move to the corresponding result of \cite[Lemma 3.2]{EGW}.

\begin{lemma}\label{lemma:approximate}
Let 
\begin{align*}
\mathcal{K}_P(w)=-\frac{1}{2}\int_{-\frac{P}{2}}^{\frac{P}{2}}wKw\ \mathrm{d}x,\quad \mathcal{N}_P(w)=-\int_{-\frac{P}{2}}^{\frac{P}{2}}N(w)\ \mathrm{d}x,
\end{align*}
\begin{align*}
 \ \mathcal{E}_P(w)=\mathcal{K}_P(w)+\mathcal{N}_P(w),
\end{align*}
and let $\{\tilde{w}_P\}$ be a bounded family of functions in $H^1(\mathbb{R})$ with $\norm{\tilde{w}_P}_{L^\infty(\mathbb{R})}\ll1$ such that
\begin{equation*}
\mathrm{supp}(\tilde{w}_P)\subset (-\frac{P}{2},\frac{P}{2}) \quad \mathrm{and}\quad  \mathrm{dist}\big(\pm \frac{P}{2},\mathrm{supp}(\tilde{w}_P)\big)\geq\frac{1}{2}P^\frac{1}{4},
\end{equation*}
and define $w_P\in H_P^1$ by the formula 
\[w_P=\sum_{j\in\mathbb{Z}}\tilde{w}_P(\cdot +jP).\]
\begin{itemize}
\item[(i)] The function $w_P$ satisfies 
\begin{equation*}
\lim_{P\rightarrow \infty}\norm{K\tilde{w}_P-Kw_P}_{H^1(-\frac{P}{2},\frac{P}{2})}=0,\quad  \lim_{P\rightarrow \infty}\norm{K\tilde{w}_P}_{H^1(\abs{x}>\frac{P}{2})}=0.
\end{equation*}
\item[(ii)] The functionals $\mathcal{E}$, $\mathcal{I}$ and $\mathcal{E}_P$, $\mathcal{I}_P$ have the properties that
\begin{equation*}
\lim_{P\rightarrow \infty}\big(\mathcal{E}(\tilde{w}_P)-\mathcal{E}_P(w_P)\big)=0,\quad \mathcal{I}(\tilde{w}_P)=\mathcal{I}_P(w_P),
\end{equation*}
and
\begin{align*}
&\lim_{P\rightarrow \infty}\norm{\mathcal{E}'(\tilde{w}_P)-\mathcal{E}'_P(w_P)}_{H^1(-\frac{P}{2},\frac{P}{2})}=0, &&\lim_{P\rightarrow \infty}\norm{\mathcal{E}'(\tilde{w}_P)}_{H^1(-\frac{P}{2},\frac{P}{2})}=0\\
&\norm{\mathcal{I}'(\tilde{w}_P)-\mathcal{I}'_P(w_P)}_{H^1(-\frac{P}{2},\frac{P}{2})}=0,&& \norm{\mathcal{I}'(\tilde{w}_P)}_{H^1(\abs{x}>\frac{P}{2})}=0.
\end{align*}
\end{itemize}
\end{lemma}

To prove Lemma \ref{lemma:approximate}, we need the following technical result of \cite[Proposition 2.1]{EGW}.
\begin{proposition}\label{proposition:technical} The linear operator $K$ satisfies\\
	(a)  $K$ belongs to \(C^\infty(H^s(\R),H^{s+|m_0|}(\R))\) \(\cap\) \(C^\infty(\mathcal{S}(\R),\mathcal{S}(\R))\) for each \(s\geq0\). \\
	(b) For each \(j\in \N\) there exists a constant \(C_l=C(\|m^{(l)}\|_{L^2(\R)})>0\) such that 
	\[
	|Kf(x)|\leq \frac{C_l\|f\|_{L^2}}{ \mathrm{dist}\big(x,\supp(f)\big)^l},\quad  x\in \R\setminus \supp (f),
	\]
	for all \(f\in L^2_c(\R)\). 
\end{proposition}

\begin{proof}[Proof of Lemma~\ref{lemma:approximate}]
The limits in (\(\romannumeral1\)) are proved in \cite[Proposition 2.1]{EGW}, so we turn to (\(\romannumeral2\)). Using (\(\romannumeral1\)) we get that $\mathcal{K}(\tilde{w}_P)-\mathcal{K}(w_P)\rightarrow 0$, as $P\rightarrow \infty$. Note that
\begin{equation}\label{estimate:N, step 1}
\begin{aligned}
\mathcal{N}(\tilde{w}_P)&=-2\int_\mathbb{R}\Psi(\tilde{w}_P)K\tilde{w}_P+\Psi(\tilde{w}_P)K(p(\tilde{w}_P))\ \mathrm{d}x\\
&=-2\int_{-\frac{P}{2}}^{\frac{P}{2}}\Psi(w_P)K\tilde{w}_P+\Psi(w_P)K(\Psi(\tilde{w}_P))\ \mathrm{d}x\\
&=-2\int_{-\frac{P}{2}}^\frac{P}{2}\Psi(w_P)K(\tilde{w}_P-w_P)+\Psi(w_P)K\big(\Psi(\tilde{w}_P)-\Psi(w_P)\big)\ \mathrm{d}x\\
&\quad+\mathcal{N}_P(w_P).
\end{aligned}
\end{equation}
In light of (\(\romannumeral1\)) we  have 
\begin{equation}\label{estimate:N, step 2}
\begin{aligned}
&\bigg|\int_{-\frac{P}{2}}^{\frac{P}{2}}\Psi(w_P)K(\tilde{w}_P-w_P)\ \mathrm{d}x\bigg|\\
&\leq \norm{\Psi(w_P)}_{L^2(-\frac{P}{2},\frac{P}{2})}\norm{K(\tilde{w}_P-w_P)}_{L^2(-\frac{P}{2},\frac{P}{2})}\rightarrow 0, \quad \text{as}\  P\longrightarrow \infty.
\end{aligned}
\end{equation}

 Since $\norm{\tilde{w}_P}_{L^\infty}\ll1$, we have $\norm{w_P}_{L^\infty}\ll1$. 
To estimate the second term on the right hand side of \eqref{estimate:N, step 1}, one first calculates
\begin{align*}
\Psi(\tilde{w}_P)-\Psi(w_P)&=\sqrt{1+\tilde{w}_P}-\sqrt{1+\sum_{j\in\mathbb{Z}}\tilde{w}_P(\cdot+jP)}+\frac{1}{2}\sum_{\abs{j}\geq 1}\tilde{w}_P(\cdot jP)\\
&=-\frac{\sum_{\abs{j}\geq 1}\tilde{w}_P(\cdot +jP)}{\sqrt{1+\tilde{w}_P}+\sqrt{1+w_P}}+\frac{1}{2}\sum_{\abs{j}\geq 1}\tilde{w}_P(\cdot +jP)\\
&=\left(\frac{1}{2}-\frac{1}{\sqrt{1+\tilde{w}_P}+\sqrt{1+w_P}}\right)\sum_{\abs{j}\geq 1}\tilde{w}_P(\cdot +jP),
\end{align*}
and then applies Proposition \ref{proposition:technical} to get
\begin{equation}\label{estimate:N, step 3}
\begin{aligned}
&\int_{-\frac{P}{2}}^\frac{P}{2}\big|K\big(\Psi(\tilde{w}_P)-\Psi(w_P)\big)\big|^2\, \mathrm{d}x\\
&\leq \int_{-\frac{P}{2}}^\frac{P}{2}\bigg|\sum_{\abs{j}\geq 1}K\left[\tilde{w}_P(\cdot +jP)\big(\frac{1}{2}-\frac{1}{\sqrt{1+\tilde{w}_P}+\sqrt{1+w_P}}\big)\right]\bigg|^2\, \mathrm{d}x\\
&\lesssim \int_{-\frac{P}{2}}^\frac{P}{2}\left(\sum_{\abs{j}\geq 1}\frac{\norm{\tilde{w}_P(\cdot +jP)\left(\frac{1}{2}-\frac{1}{\sqrt{1+\tilde{w}_P}+\sqrt{1+w_P}}\right)}_{L^2(-\frac{P}{2},\frac{P}{2})}}{\text{dist}\big(x+jP,\text{supp}(\tilde{w}_P)\big)^3}\right)^2\, \mathrm{d}x\\
&\lesssim \|\tilde{w}_P\|_{L^2}\int_{-\frac{P}{2}}^\frac{P}{2}\big(\sum_{\abs{j}\geq 1}\frac{1}{(jP+\frac{1}{2}P^{\frac{1}{4}})^3}\big)^2\, \mathrm{d}x\\
&\rightarrow 0, \quad \text{as}\  P\longrightarrow \infty.
\end{aligned}
\end{equation}
Hence we obtain 
\begin{equation}\label{estimate:N, step 4}
\begin{aligned}
&\bigg|\int_{-\frac{P}{2}}^\frac{P}{2}\Psi(w_P)K\big(\Psi(\tilde{w}_P)-\Psi(w_P)\big)\ \mathrm{d}x\bigg|\\
&\leq\norm{\Psi(w_P)}_{L^2(-\frac{P}{2},\frac{P}{2})}\norm{K\big(\Psi(\tilde{w}_P)-\Psi(w_P)\big)}_{L^2(-\frac{P}{2},\frac{P}{2})}\rightarrow 0, \quad \text{as}\  P\longrightarrow \infty.
\end{aligned}
\end{equation} 
From \eqref{estimate:N, step 1}, \eqref{estimate:N, step 2} and \eqref{estimate:N, step 4}, it follows that $\mathcal{N}(\tilde{w}_P)-\mathcal{N}_P(w_P)\rightarrow 0$, which in turn implies that 
\[\mathcal{E}(\tilde{w}_P)-\mathcal{E}_P(w_P)\rightarrow 0, \quad \text{as}\  P\longrightarrow \infty.\] 

The equality \(\mathcal{I}(\tilde{w}_P)=\mathcal{I}_P(w_P)\) is immediate.  

A direct calculation yields 
\begin{equation*}
\mathcal{N}'(w)=-\left(\frac{1}{\sqrt{1+w}}-1\right)Kw-\frac{2}{\sqrt{1+w}}K(\Psi(w)),
\end{equation*}
so we may estimate
\begin{align*}
&\norm{\mathcal{N}'(\tilde{w}_P)-\mathcal{N}'_P(w_P)}_{L^2(-\frac{P}{2},\frac{P}{2})}\\
&\leq \norm{\left(\frac{1}{\sqrt{1+w_P}}-1\right)(K\tilde{w}_P-Kw_P)}_{L^2(-\frac{P}{2},\frac{P}{2})}\\
&\quad +\norm{\frac{2}{\sqrt{1+w_P}}K\big(\Psi(\tilde{w}_P)-\Psi(w_P)\big)}_{L^2(-\frac{P}{2},\frac{P}{2})}\\
&\rightarrow 0,  \quad \text{as}\  P\longrightarrow \infty, 
\end{align*}
where we have used (\(\romannumeral1\)) and \eqref{estimate:N, step 3}. One can similarly show that 
\begin{equation*}
\norm{\frac{\diff}{\diff x}\mathcal{N}'(\tilde{w}_P)-\frac{\diff}{\diff x}\mathcal{N}'_P(w_P)}_{L^2(-\frac{P}{2},\frac{P}{2})}\rightarrow 0,\quad \text{as}\  P\rightarrow \infty.
\end{equation*}
Hence
\begin{equation*}
\norm{\mathcal{E}'(\tilde{w}_P)-\mathcal{E}'_P(w_P)}_{H^1(-\frac{P}{2},\frac{P}{2})}\rightarrow 0,\quad \text{as}\  P\rightarrow \infty.
\end{equation*}

Note that $\frac{1}{\sqrt{1+\tilde{w}_P}}-1=0$ for $\abs{x}>\frac{P}{2}$, we calculate 
\begin{align*}
&\norm{\mathcal{N}'(\tilde{w}_P)}_{L^2(\abs{x}>\frac{P}{2})}\\
&=\norm{\left(\frac{1}{\sqrt{1+\tilde{w}_P}}-1\right)K\tilde{w}_P+\frac{2}{\sqrt{1+\tilde{w}_P}}K(\Psi(\tilde{w}_P))}_{L^2(\abs{x}>\frac{P}{2})}\\
&=\norm{\frac{2}{\sqrt{1+\tilde{w}_P}}K(\Psi(\tilde{w}_P))}_{L^2(\abs{x}>\frac{P}{2})}.
\end{align*}
Since $\supp(\Psi(\tilde{w}_P))=\supp(\tilde{w}_P)$,  we have \(\norm{K(\Psi(\tilde{w}_P))}_{L^2(\abs{x}>\frac{P}{2})}\rightarrow 0\).
It follows that 
\[\norm{\mathcal{N}'(\tilde{w}_P)}_{L^2(\abs{x}>\frac{P}{2})}\rightarrow 0,\quad \text{as}\  P\rightarrow \infty.\]
 A similar calculation shows that 
 \[\norm{\frac{\diff}{\diff x}\mathcal{N}'(\tilde{w}_P)}_{L^2(\abs{x}>\frac{P}{2})}\rightarrow 0.\]
Consequently, we have \[\norm{\mathcal{N}'(\tilde{w}_P)}_{H^1(\abs{x}>\frac{P}{2})}\rightarrow 0, \quad \text{as}\  P\rightarrow \infty.\]
\end{proof}

Just as in \cite[Theorem 6.3]{EGW} we can relate the minimizers of $\mathcal{E}$ with those of $\mathcal{E}_{lw}$: 
\begin{equation*}
\sup_{w\in D_q}\text{dist}_{H^{j_*}(\mathbb{R})}(S_{lw}^{-1}w,D_{lw})\rightarrow 0,\quad \text{ as }q\rightarrow 0,
\end{equation*}
where $D_{lw}$ is the set of minimizers of $\mathcal{E}_{lw}$ over the set 
\[\{w\in H^{j_*}(\mathbb{R}):\ \mathcal{I}(w)=1\},\]
and $D_q$ is the set of minimizers of $\mathcal{E}$ over $V_{q,R}$. 
%\section{Preliminaries}
%Here we state some general results about the operator $K$, the functional $\mathcal{E}$ and solutions of \eqref{vareq}. The proofs of these results are omitted since they differ nothing or very little from the corresponding results of [EGW]. 
%\begin{proposition}\label{proposition:multiplier 1}
%(i) The linear operator $K$ belongs to \(C^\infty(H^s(\R),H^{s+|m_0|}(\R))\) \(\cap\) \(C^\infty(\mathcal{S}(\R),\mathcal{S}(\R))\) for each \(s\geq0\). \\
%(ii) For each \(j\in \N\) there exists a constant \(C_j>0\) such that 
%\[
%|Kf(x)|\leq \frac{C_j\|f\|_{L^2}}{ dist(x,\supp(f))^j},\quad  x\in \R\setminus \supp (f),
%\]
%for all \(f\in L^2_c(\R)\). 
%\end{proposition}
%Similarly, it is not difficult to verify the following results.
%\begin{proposition}
%(i) The functionals $\mathcal{K}$, $\mathcal{N}$ and $\mathcal{I}$ belong to $C^1(U,\mathbb{R})$ and their $L^2(\mathbb{R})$ are given by the formulae
%\begin{equation*}
%\mathcal{K}'(w)=-Kw,\ \mathcal{N}'(w)=-n(w):=-\left(\frac{1}{\sqrt{1+w}}-1\right)Kw-\frac{2}{\sqrt{1+w}}K(p(w)),\ \mathcal{I}'(w)=w.
%\end{equation*}
%(ii) The functional $\mathcal{E}$ belongs to $C(H^s(\mathbb{R}),\mathbb{R})$ for each $s>1/2$.
%\end{proposition} 
%\begin{lemma}
%For sufficiently small values of $R$, every solution $w\in U$ of \eqref{vareq} belongs to $H^{k+1}(\mathbb{R})$ and satisfies 
%\begin{equation*}
%\norm{w}_{H^{k+1}}<c\norm{w}_{H^1}.
%\end{equation*}
%\end{lemma}

We finally include a regularity result for the travelling wave solutions of \eqref{vareq} which corresponds to \cite[Lemma 2.3]{EGW}.
\begin{lemma}\label{le:regularity} Let \(w\) be a solution of \eqref{vareq} in with $\norm{w}_{L^\infty}\ll1$. Then for any \(k\in \mathbb{N}_+\), \(w\in H^{k}\) and satisfies
	\[
	\|w\|_{H^k}\leq C(k,\|w\|_{H^1}).   
	\]   
\end{lemma}
\begin{proof} Let 
	\[
	f=\sqrt{1+w}-1,
	\] 
then one has \(\|f\|_{L^\infty}\ll1\) due to $\norm{w}_{L^\infty}\ll1$. In view of \eqref{vareq}, \(f\) solves 
	\begin{equation}\label{eqf}
	f=\frac{2}{\lambda (1+f)(2+f)}Kf. 
	\end{equation}
	Differentiating in \eqref{eqf} yields 
	\begin{align}\label{eq: direvative}
	\partial_xf=\frac{2}{\lambda [(1+f)(2+f)+f(2+f)+f(1+f)]}K\partial_xf. 
	\end{align}
	The denominator is positive due to \(\|f\|_{L^\infty}\ll1\).

	Let \(l\in \{1,2,\cdots,k\}\). For each fixed \(f\in H^l\) we define a formula \(\phi_f\) by 
	\[
	\phi_f(g)=\frac{2}{\lambda [(1+f)(2+f)+f(2+f)+f(1+f)]}g.
	\]
	Then one now may follow the argument in [EGW, Lemma 2.3] by using the properties of \(\phi_f\) and \(K\) to show  
	\begin{align*}
	\|\partial_xf\|_{H^l}\leq C(\|f\|_{H^1})\|\partial_xf\|_{L^2}.  
	\end{align*}
	
	For completeness, we give its proof here. For any \(s\in [0,l]\),  it is easy to see that \(\phi_f\) and \(K\) define an operator in \(B(H^s,H^s)\) and  \(B(H^s,H^{s+|m_0|})\), respectively. Thus the composition 
	\[
	\psi_f=\phi_f\circ K\in B(H^s,H^{s_*}),\quad s_*=\min\{l,s+|m_0|\},
	\]
	and the norm of \(\psi_f\) depends upon \(\|f\|_{H^l}\). 
	Consequently, any solution \(g\) of \(g=\psi_f(g)\) belongs to \(H^{s_*}\) and satisfies  
	\[
	\|g\|_{H^{s_*}}\leq C_{l,\|f\|_{H^l}}\|g\|_{H^s}.    
	\] 
	Applying this argument recursively, one finds that any solution \(g\in L^2\) belongs to \(H^l\) and 
	satisfies  
	\[
	\|g\|_{H^l}\leq C(l,\|f\|_{H^l})\|g\|_{L^2}.    
	\] 
	
	Since \eqref{eq: direvative} is equivalent to \(\partial_xf=\psi_f(\partial_xf)\), a bootstrap argument 
	shows that \(f^{\prime}\in H^l\) with  
	\[
	\|\partial_xf\|_{H^l}\leq C(l,\|f\|_{H^1})\|\partial_xf\|_{L^2},\ l=1,2,\cdots,k. 
	\] 
So far we have shown that 
\[
\|f\|_{H^k}\leq C(k,\|f\|_{H^1}).   
\] 	
	Finally, recalling that \(w=f^2+2f\) and \(H^l\) is an algebra, we therefore obtain  
	\[
	\|w\|_{H^k}\leq C(k,\|f\|_{H^1})\leq C(k,\|w\|_{H^1}),  
	\] 
	where we have used \(\|w\|_{L^\infty}\ll1\) in the last inequality.

\end{proof}
\begin{remark}
The results of the present work may be extended to a more general class of nonlinearities $N$. On the one hand, we have that the leading order part of $N$ is cubic, but this could be extended to higher power nonlinearities.
%Also, there is no need for the nonlocal operator apperaing in $N$ to have symbl eqaul to $m$, 
On the other hand, the multiplier operator $K$ appearing in $N$ can be replaced by an operator $K'$ belonging to a wider class of Fourier multipliers. For instance, it is not necessary for the symbol of this $K'$ to be of negative order. An example is $K'=\text{Id}$, which yields the nonlinearities studied in \cite{EGW}.   
%for instance, the operator appearing in $N$ may have symbol of order zero. An example of this would be the identity operator, which yields the class of operators studied in \cite{EGW}. 
\end{remark}

\section{Acknowledgment} Both authors would like to thank M. Ehrnstr{\"o}m and E. Wahl{\'e}n for suggesting this topic.

\bibliographystyle{siam}

\begin{thebibliography}{10}

\bibitem{MR2991247}
{\sc P.~Aceves-S\'{a}nchez, A.~A. Minzoni, and P.~Panayotaros}, {\em Numerical
  study of a nonlocal model for water-waves with variable depth}, Wave Motion,
  50 (2013), pp.~80--93.

\bibitem{MR3603270}
{\sc G.~Bruell, M.~Ehrnstr\"{o}m, and L.~Pei}, {\em Symmetry and decay of
  traveling wave solutions to the {W}hitham equation}, J. Differential
  Equations, 262 (2017), pp.~4232--4254.

\bibitem{Buffoni2004}
{\sc B.~Buffoni}, {\em {Existence and conditional energetic stability of
  capillary-gravity solitary water waves by minimisation}}, Archive for
  Rational Mechanics and Analysis, 173 (2004), pp.~25--68.

\bibitem{MR3844340}
{\sc J.~D. Carter}, {\em Bidirectional {W}hitham equations as models of waves
  on shallow water}, Wave Motion, 82 (2018), pp.~51--61.

\bibitem{doi:10.1111/sapm.12221}
{\sc K.~M. Claassen and M.~A. Johnson}, {\em Numerical bifurcation and spectral
  stability of wavetrains in bidirectional whitham models}, Studies in Applied
  Mathematics, 141, pp.~205--246.

\bibitem{ED}
{\sc E.~Dinvay}, {\em On well-posedness of a dispersive system of the
  {W}hitham–-{B}oussinesq type}, Applied Mathematics Letters, 88 (2018),
  pp.~13--20.

\bibitem{DDK}
{\sc E.~Dinvay, D.~Dutykh, and H.~Kalisch}, {\em A comparative study of
  bi-directional whitham systems}, submitted for publication,  (2018).

\bibitem{MR3564304}
{\sc V.~Duch\^{e}ne, S.~Israwi, and R.~Talhouk}, {\em A new class of two-layer
  {G}reen-{N}aghdi systems with improved frequency dispersion}, Stud. Appl.
  Math., 137 (2016), pp.~356--415.

\bibitem{MR3841973}
{\sc V.~Duch\^{e}ne, D.~Nilsson, and E.~Wahl\'{e}n}, {\em Solitary {W}ave
  {S}olutions to a {C}lass of {M}odified {G}reen--{N}aghdi {S}ystems}, J. Math.
  Fluid Mech., 20 (2018), pp.~1059--1091.

\bibitem{EGW}
{\sc M.~Ehrnstr{\"{o}}m, M.~D. Groves, and E.~Wahl{\'{e}}n}, {\em {On the
  existence and stability of solitary-wave solutions to a class of evolution
  equations of Whitham type}}, Nonlinearity, 25 (2012), pp.~2903--2936.

\bibitem{Ehrnström2018}
{\sc M.~Ehrnstr{\"o}m, M.~A. Johnson, and K.~M. Claassen}, {\em Existence of a
  highest wave in a fully dispersive two-way shallow water model}, Archive for
  Rational Mechanics and Analysis,  (2018).

\bibitem{MR2555644}
{\sc M.~Ehrnstr\"{o}m and H.~Kalisch}, {\em Traveling waves for the {W}hitham
  equation}, Differential Integral Equations, 22 (2009), pp.~1193--1210.

\bibitem{2017arXiv170804551E}
{\sc M.~{Ehrnstr{\"o}m}, L.~{Pei}, and Y.~{Wang}}, {\em {A conditional
  well-posedness result for the bidirectional Whitham equation}}, arXiv:
  1708.04551, 2017.

\bibitem{2016arXiv160205384E}
{\sc M.~{Ehrnstr\"om} and E.~{Wahl{\'e}n}}, {\em {On Whitham's conjecture of a
  highest cusped wave for a nonlocal dispersive equation}}, arXiv: 1602.05384,
  2016.

\bibitem{MR2847283}
{\sc M.~D. Groves and E.~Wahl\'{e}n}, {\em On the existence and conditional
  energetic stability of solitary gravity-capillary surface waves on deep
  water}, J. Math. Fluid Mech., 13 (2011), pp.~593--627.

\bibitem{MR3749383}
{\sc V.~M. Hur and L.~Tao}, {\em Wave breaking in a shallow water model}, SIAM
  J. Math. Anal., 50 (2018), pp.~354--380.

\bibitem{2018arXiv180504372K}
{\sc H.~{Kalisch} and D.~{Pilod}}, {\em {On the local well-posedness for a full
  dispersion Boussinesq system with surface tension}}, arXiv: 1805.04372, 2018.

\bibitem{MR3763731}
{\sc C.~Klein, F.~Linares, D.~Pilod, and J.-C. Saut}, {\em On {W}hitham and
  related equations}, Stud. Appl. Math., 140 (2018), pp.~133--177.

\bibitem{MR3060183}
{\sc D.~Lannes}, {\em The water waves problem}, vol.~188 of Mathematical
  Surveys and Monographs, American Mathematical Society, Providence, RI, 2013.
\newblock Mathematical analysis and asymptotics.

\bibitem{MR778974}
{\sc P.-L. Lions}, {\em The concentration-compactness principle in the calculus
  of variations. {T}he locally compact case. {II}}, Ann. Inst. H. Poincar\'{e}
  Anal. Non Lin\'{e}aire, 1 (1984), pp.~223--283.

\bibitem{MR3390078}
{\sc D.~Moldabayev, H.~Kalisch, and D.~Dutykh}, {\em The {W}hitham equation as
  a model for surface water waves}, Phys. D, 309 (2015), pp.~99--107.

\bibitem{MR3668593}
{\sc J.-C. Saut, C.~Wang, and L.~Xu}, {\em The {C}auchy problem on large time
  for surface-waves-type {B}oussinesq systems {II}}, SIAM J. Math. Anal., 49
  (2017), pp.~2321--2386.

\bibitem{arXiv:1802.10040}
{\sc A.~Stefanov and D.~Wright}, {\em {Small amplitude traveling waves in the
  full-dispersion Whitham equation}}, arXiv:1802.10040.

\bibitem{MR3523508}
{\sc S.~Trillo, M.~Klein, G.~F. Clauss, and M.~Onorato}, {\em Observation of
  dispersive shock waves developing from initial depressions in shallow water},
  Phys. D, 333 (2016), pp.~276--284.

\bibitem{MR0671107}
{\sc G.~B. Whitham}, {\em Variational methods and applications to water waves},
   (1970), pp.~153--172.

\end{thebibliography}

\end{document}